\documentclass[12pt]{amsart}
\usepackage{graphicx}
\usepackage{amstext}
\usepackage{amsfonts, amssymb}
\usepackage{amsmath}
\usepackage{color}
\usepackage{latexsym}
\usepackage{indentfirst}
\usepackage{amssymb}
\usepackage[utf8]{inputenc}

\advance\textheight by \topskip
\newtheorem{theorem}{Theorem}
\newtheorem{proposition}[theorem]{Proposition}
\newtheorem{lemma}[theorem]{Lemma}

\theoremstyle{definition}

\title{\hspace{.4cm}
First eigenvalue of the Laplacian on compact surfaces 
for large genera}
\author{Antonio Ros}     
\thanks{This work is supported in part by the IMAG–Maria de Maeztu grant CEX2020-001105-M / AEI / 10.13039/501100011033, 
MICINN grant PID2020-117868GB-I00 and Junta de Andalucıa grant P18-FR4049.}

\begin{document}

\begin{abstract}
{\small 
For any Riemannian metric $ds^2$ on a compact surface of genus $g$, 
Yang and Yau proved that the normalized first eigenvalue of the Laplacian $\lambda_1(ds^2)Area(ds^2)$ is bounded in terms of the genus.
In particular, if  $\Lambda_1(g)$ is the supremum for each $g$,
it follows that the asymptotic growth of the sequence 
${\Lambda_1(g)}$ is no larger than the one of  $4\pi g$.
In this paper we improve the result and we show that
\[
\limsup_{g\, \rightarrow\, \infty} 
\,
\frac{1}{g}\Lambda_1(g) \leq 4(3-\sqrt{5})\pi \approx 3.056\pi.
\]
}

\vspace{.1cm}

\end{abstract}

\maketitle

\section{Introduction}

Let $\Sigma$ a compact orientable surface of genus $g$ and $ds^2$ be a Riemannian metric on it.  The Laplacian 
of the metric is a fundamental operator associated to $ds^2$ and its eigenvalues are given by a divergent sequence 
\[
0=\lambda_0(ds^2) < \lambda_1(ds^2)  \leq \lambda_2(ds^2) \leq \cdots,
\]
where each eigenvalue is counted with multiplicity. The normalized first eigenvalue $ \lambda_1(ds^2)Area(ds^2)$ is invariant under homothetic rescaling and   we define
\[
\Lambda_1(g) = \sup \left\{  \lambda_1(ds^2)Area(ds^2)\, \big{/}  ds^2 \,\, {\rm is\,\, a \,\,Riemannian \,\,metric\,\,on \,\,\Sigma }  \right\}.  
\]
Yang and Yau \cite{yy} gave the upper bound 
\begin{equation}
\Lambda_1(g) \leq \left[\frac{g+3}{2}\right]8\pi,
\label{yy}
\end{equation}
where $\left[x\right]$ denotes the integer part of $x$. Their argument uses as test functions branched conformal maps between $(\Sigma,ds^2)$ and the 
unit sphere $S^2(1)\subset \mathbb{R}^3$. 
This idea was first used by Hersch \cite{hersch} who proved that $\Lambda_1(0)= 8\pi$, the equality holding only for constant curvature metrics, and  later by \cite{bourgui,liyau} and other authors. 

\vspace{.1cm}

Concerning the invariant $\Lambda_1(g)$, among other known results we remark
Nadirashvili \cite{nadir}, who proved that  $\Lambda_1(1)=\frac{8}{\sqrt{3}}\pi \approx 14.510\pi$, where the equality holds for the flat equilateral metric,
and Nayatani and Shoda \cite{nayatani}, $\Lambda_1(2)=16\pi$, the equality holding for the Bolza spherical surface and some other branched spherical metrics.

\vspace{.1cm}
Regarding the metrics   that maximize the functional $\Lambda_1$,
for any genus $g$ we have an {\it extremal metric} $ds^2$ 
with conical singularities on $\Sigma$ such that $\lambda_1(ds^2)Area(ds^2)= \Lambda_1(g)$ and a branched minimal immersion 
$f:\Sigma\longrightarrow S^m(1)$, $m\geq 2$, 
by the first eigenfunctions of $ds^2$, see  Petrides \cite{Petri} and Matthiesen and Siffert \cite{matthie}.
For other related results see  Cianci, Karpukhin and Medvedev \cite{karpu1}, El Soufi and Ilias \cite{elsoufi},  
Fraser and Schoen \cite{fraser}, Gomyou and Nayatani \cite{gnayatani}, 
Montiel and Ros \cite{montiel} and  Nadirashvili and Sire \cite{nadir1}.

\vspace{.2cm}

We next briefly explain the results of this paper. Bourguignon, Li and Yau \cite{bourgui} gave a first eigenvalue upper bound for algebraic Kaehler manifolds which extends \cite{yy} (we will only consider the case of holomorphic curves in the complex projective space) whose main idea is as follows. There is a natural embedding of the complex projective space  $\mathbb{C}{\mathrm P}^n$ with the Fubini-Study metric in the Euclidean space ${H\hspace{-.03cm}M}_1\hspace{-.03cm}(n+1)$ of square Hermitian matrices of order $n+1$ and trace $1$. 
A compact Riemannian surface $(\Sigma,ds^2)$ will be seen as a Riemann surface $\Sigma$ with a conformal metric $ds^2$ on it. 
A full holomorphic map in the complex projective space 
$A:\Sigma\longrightarrow\mathbb{C}{\mathrm P}^n\subset {H\hspace{-.03cm}M}_1\hspace{-.03cm}(n+1)$ defines a complex curve and its  {\it energy} depends  
only of the degree of the curve. By composing  with projective transformations of $\mathbb{C}{\mathrm P}^n$ 
one can get the center of mass of $A$ to be proportional to the identity matrix
and, thus, they obtained that $\lambda_1(ds^2)Area(ds^2)$ is bounded in term of this degree.   

\vspace{.1cm}
Lately Ros \cite{ros2} and Karpukhin and Vinokurov \cite{karpu} further elaborate on this idea by considering the map 
$\phi_a=A+aB:\Sigma\longrightarrow  {H\hspace{-.03cm}M}_1\hspace{-.03cm}(n+1)$, where $a\in\mathbb{R}$ and $B$ is a kind of {\it Gauss map} of the curve
(i.e. $B$ is a spherical map depending on the tangent lines of the curve). 
The energy of $\phi_a$ is a function that depends on  $a$, the {\it degree} of the curve, the genus of $\Sigma$ and the number of branch points (if any). 
Using projective transformations one can control, for certain values of $a$, the  center of mass of $\phi_a$ and in this way obtain new upper bounds for the
normalized  first eigenvalue of $ds^2$. In accordance with this, as any genus-three Riemann surface is either hyperelliptic or a quartic curve in the complex 
projective plane,  one proves that 
 \[
 \Lambda_1(3)  \leq 16(4-\sqrt{7})\pi \approx 21.668\pi,
 \]
   \cite{ros2}  (the corresponding bound in \cite{yy} is $24\pi$).
In \cite{karpu}, using holomorphic maps provided by the Brill-Noether theory, one improves the bound of  Yang and Yau for all genera 
$g \neq  4,6,8,14$.

\vspace{.1cm}

A fundamental step in the proof of the above results is to be able to control the center of mass of the  test  functions.
In this paper we obtain a complete solution for the maps \mbox{$\phi_a:\Sigma\longrightarrow {H\hspace{-.03cm}M}_1\hspace{-.03cm}(n+1)$}.
In Proposition \ref{centro}, using the Algebraic Geometry of the curve $A:\Sigma\longrightarrow  \mathbb{C}{\mathrm P}^n$ and some ideas 
that originate in Ros \cite{ros1,ros2}, we prove  that
\begin{quote}
{\it $\forall\, a$, $0\leq a < 1/2$, after composing with a suitable projective transformation, the center of mass of the map $\phi_{a}$ 
becomes the point $\frac{1}{n+1}I$.}
\end{quote}
Using  this property we give in Theorem \ref{teor2} an inequality for the supremun  first eigenvalue $\Lambda_1(g)$ which improves 
\cite{karpu}.

\vspace{.1cm}

As a main application we focus on the asymptotic behavior of the sequence $\{\Lambda_1(g)\}$.
From (\ref{yy})  we have that the upper limit of $\Lambda_1(g)/g$ is $\leq 4\pi$ and in \cite{karpu} the authors improve this inequality until 
$\leq 3.416\pi$. 
In the following statement we provide a better upper bound which summarizes our contribution, see Theorem \ref{teor3}:

\vspace{.1cm}

\begin{quote}{\it The sequence $\big\{ \Lambda_1(g)\big\}_g$ of extremal normalized  first  eigenvalues of the Laplacian on compact surfaces 
satisfies the following growth properties}
\vspace{-.1cm}
\[
\pi \leq \limsup_{g\rightarrow\infty} \frac{1}{g}\Lambda_1(g) \leq 4(3-\sqrt{5})\pi \approx 3.056\pi.
\label{main}
\]
\end{quote}
The lower bound has been obtained recently by Hide and Magee \cite{hide} (they use  
 a family of hyperbolic metrics and improve previous known bounds in that context).

\pagebreak

\section{Preliminaries} 

\subsection{The complex projective space.} 

\hspace{.1cm}

\hspace{.2cm}

Let $H\hspace{-.03cm}M\hspace{-.03cm}(n+1)=\{A\in gl(n+1,\mathbb{C}) \, / \, \bar{A}=A^t \}$ the space of 
$(n+1)\hspace{-.08cm}\times \hspace{-.08cm}(n+1)$-Hermitian matrices with the Euclidean metric 
\[
\langle A,B\rangle = 2\hspace{.05cm} tr\hspace{.01cm} AB \hspace{.5cm} \forall A,B\in H\hspace{-.03cm}M\hspace{-.03cm}(n+1)
\]
and ${H\hspace{-.03cm}M}_1\hspace{-.03cm}(n+1)=\{A\in H\hspace{-.03cm}M\hspace{-.03cm}(n+1)\, /\,  tr\hspace{.01cm} A=1\}\simeq\mathbb{R}^{n(n+2)}$ 
be the  hyperplane given by the trace $1$ restriction. Along the paper, a clever element of this hyperplane will be the point $\frac{1}{n+1}I$, $I$  
being the identity matrix.

\vspace{.2cm}

The submanifold $\mathbb{C}{\mathrm P}^n=\{A\in {H\hspace{-.03cm}M}_1\hspace{-.03cm}(n+1)\, / \, AA=A \}$ with the induced metric is isometric to the complex projective space with the Fubini-Study metric of constant holomorphic sectional curvature $1$, see Ros \cite{ros1}. 

\vspace{.2cm}

The action of the unitary group $U(n+1)$ on $\mathbb{C}{\mathrm P}^n$ is given by $(P,A)\mapsto \bar{P}^{t}\hspace{-.1cm}AP$, 
where $P\in U(n+1)$ and $A\in \mathbb{C}{\mathrm P}^n$. Hence the embedding of $\mathbb{C}{\mathrm P}^n$ in 
${H\hspace{-.03cm}M}_1\hspace{-.03cm}(n+1)$ is $U(n+1)$-equivariant. 

\vspace{.2cm}

For any $A$ in $\mathbb{C}{\mathrm P}^n $ the tangent space at that point identified with a subspace of 
${H\hspace{-.03cm}M}\hspace{-.03cm}(n+1)$  is given by 
$T_A \mathbb{C}{\mathrm P}^n= \{X \in {H\hspace{-.03cm}M}\hspace{-.03cm}(n+1) \, / \, XA+AX=X\}$ and 
 the second fundamental form $\tilde{\sigma}$ of $\mathbb{C}{\mathrm P}^n \subset 
{H\hspace{-.03cm}M}_1\hspace{-.03cm}(n+1)$ maps  
the tangent vectors $X,Y\in T_A \mathbb{C}{\mathrm P}^n$ into 
$\tilde{\sigma}(X,Y) \in  T_A^\perp \mathbb{C}{\mathrm P}^n.$

\vspace{.1cm}

Among properties of the embedding we remark the following ones, \cite{ros1}:

\begin{itemize} 
  \item[$\cdot$] Complex projective lines $\mathbb{C}{\mathrm P}^1\subset \mathbb{C}{\mathrm P}^n$ are totally geodesic and,  when viewed in ${H\hspace{-.03cm}M}_1\hspace{-.03cm}(n+1)$, are given by round 2-spheres of radius one. 
\vspace{.1cm}
 \item[$\cdot$] If $J$ is the complex structure in $\mathbb{C}{\mathrm P}^n$, then 
  $\tilde{\sigma}(JX,JY)=\tilde{\sigma}(X,Y)$,  $\forall \,X,Y\in T_A \mathbb{C}{\mathrm P}^2$.
 \vspace{.1cm} 
 \item[$\cdot$]  $\widetilde{\nabla} \tilde{\sigma}=0$, i. e. the second fundamental form is parallel.
\vspace{.1cm}
\item[$\cdot$]  $\mathbb{C}{\mathrm P}^n$ is a minimal submanifold in the sphere $S^{n(n+2)-1}$ of 
${H\hspace{-.03cm}M}_1\hspace{-.03cm}(n+1)$ of center $\frac{1}{n+1}I$ and radius $\sqrt{\frac{2n}{n+1}}$.
\end{itemize}
  
\vspace{.1cm}

Now we consider {\it the convex hull}  
of the complex projective space in ${H\hspace{-.03cm}M}_1\hspace{-.03cm}(n+1)$. 
 If $A\in HM(n+1)$ we will use the notation
$A > 0$, resp. $A\geq 0$, when the Hermitian matrix is positive definite, resp. positive semidefinite. 
Then the convex hull ${\mathcal H}$ of $\mathbb{C}{\mathrm P}^n$ in ${H\hspace{-.03cm}M}_1\hspace{-.03cm}(n+1)$
verifies the following properties, see \cite{bourgui}, \cite{ros2}:

\vspace{.1cm}
\begin{itemize}
\item[$\cdot$] ${\mathcal H} = \{A\in {H\hspace{-.03cm}M}_1\hspace{-.03cm}(n+1)\, / \, A\geq 0 \}$, 

\vspace{.2cm}

\item[$\cdot$]   
$int\,{\mathcal H}= \{ A\in {\mathcal H} \, / \, rank\hspace{.02cm} A =n+1 \} \, = \, 
\{ A\in {H\hspace{-.03cm}M}_1\hspace{-.03cm}(n+1) \, / \, A> 0\}$, where by $int\,{\mathcal H}$ 
we denote the topological interior of ${\mathcal H}$ in ${H\hspace{-.03cm}M}_1\hspace{-.03cm}(n+1)$, 

\vspace{.2cm}

\item[$\cdot$] $\partial{\mathcal H} = \{A\in {\mathcal H} \, / \, rank \hspace{.02cm} A \leq n \}$,

\vspace{.2cm}

\item[$\cdot$] $\mathbb{C}{\mathrm P}^n= \{A\in {\mathcal H} \, / \, rank \hspace{.02cm} A=1\}$,

\vspace{.2cm}

\item[$\cdot$]   $\{A\in {\mathcal H} \, / \, rank \hspace{.02cm} A\leq 2\}$ is the union of  all the unit $3$-balls in ${H\hspace{-.03cm}M}_1\hspace{-.03cm}(n+1)$ enclosed by complex projective lines $\mathbb{C}{\mathrm P}^1 \subset \mathbb{C}{\mathrm P}^n$.
\end{itemize}

\vspace{.2cm}

The projection of $\mathbb{C}^{n+1}-\{0\}$ over $\mathbb{C}{\mathrm P}^n$ 
\vspace{-.1cm}
\[
z\longmapsto \frac{1}{|z|^2}\bar{z}^t z,
\]
with $z=(z_0,z_1, \ldots, z_n)$, defines the identification between $\mathbb{C}^{n+1}\hspace{-.1cm}-\{0\}/\hspace{-.15cm}\sim$, the usual projective 
space with homogeneous coordinates $[z]=[z_0, z_1,\ldots, z_n]$, and $\mathbb{C}{\mathrm P}^n$ viewed as submanifold of 
${H\hspace{-.03cm}M}_1\hspace{-.03cm}(n+1)$. Along this paper both views of the projective space will be used at our convenience.

\vspace{.1cm}

Given a regular matrix $P\in GL(n+1,\mathbb{C})$, the projective transformation $f_P:\mathbb{C}{\mathrm P}^n\longrightarrow \mathbb{C}{\mathrm P}^n$,
is given by
$[z]\longmapsto [zP]$,  i. e. $[z_0,z_1,\ldots,z_n]\longmapsto [(z_0,z_1,\ldots,z_n)P]$. When we describe it in terms of elements of $\mathbb{C}{\mathrm P}^n$ we have
\begin{equation}
\label{projectivity}
f_P:\frac{1}{|z|^2}\bar{z}^t z \longmapsto \frac{1}{|zP|^2}\bar{P}^t\bar{z}^t zP.
\end{equation}

In this paper, as in  \cite{bourgui}, we consider projective transformations $f_P$, where 
$P$ is {a} positive definite { matrix} in $HM(n+1)$: Any other projectivity can be decomposed as  $f_P\circ f_Q$, 
where $P>0$ and $Q$ is a unitary matrix.
Moreover, after multiplying by a positive scalar factor,
we assume that $trace\hspace{.03cm} P=1$. Therefore, up to unitary motions, the space of projective transformations is parametrized by the interior of the convex hull of $\mathbb{C}{\mathrm P}^n$,
\[
f_P:\mathbb{C}{\mathrm P}^n\longrightarrow \mathbb{C}{\mathrm P}^n, \hspace{1cm} P\in int\,{\mathcal H}.
\]  

\vspace{.2cm} 

When $P\in \partial\hspace{.03cm} {\mathcal H}$, then $P$  it is not a regular matrix but we will still interested into the associated projection map $f_P$ defined as follows. 
Let ${\mathcal R}, {\mathcal S}\subset \mathbb{C}{\mathrm P}^n$ be the subspaces determinate by the linear subspaces $Kernel(P), Image(P)\subset \mathbb{C}^{n+1}$, respectively. Then
${\mathcal R}\cap {\mathcal S}=\emptyset$, $\dim {\mathcal R} =n-rank\, P$ and $\dim {\mathcal S} =rank\, P -1$.
The map $f_P:[z]\longmapsto [zP]$ is defined over $\mathbb{C}{\mathrm P}^n$ minus $\mathcal R$, its image is $\mathcal S$,
\begin{equation}
f_P:\mathbb{C}{\mathrm P}^n\hspace{-.03cm}-\hspace{-.03cm}{\mathcal R}\longrightarrow {\mathcal S}
\label{proyeccion}
\end{equation} 
and $P$ belongs to the convex hull of $\mathcal S$ in 
${H\hspace{-.03cm}M}_1\hspace{-.03cm}(n+1)$.
 It can be see as the classic projection map since ${\mathcal R}$ over  ${\mathcal S}$, see Figure \ref{111}, followed by a linear projectivity of the image subspace, 
 ${f_P}_{\big| {\mathcal S}}\hspace{-.03cm}:\hspace{-.02cm}{\mathcal S}\longrightarrow {\mathcal S}$.
\begin{figure}[h]
\begin{center}
\includegraphics[width=8cm]{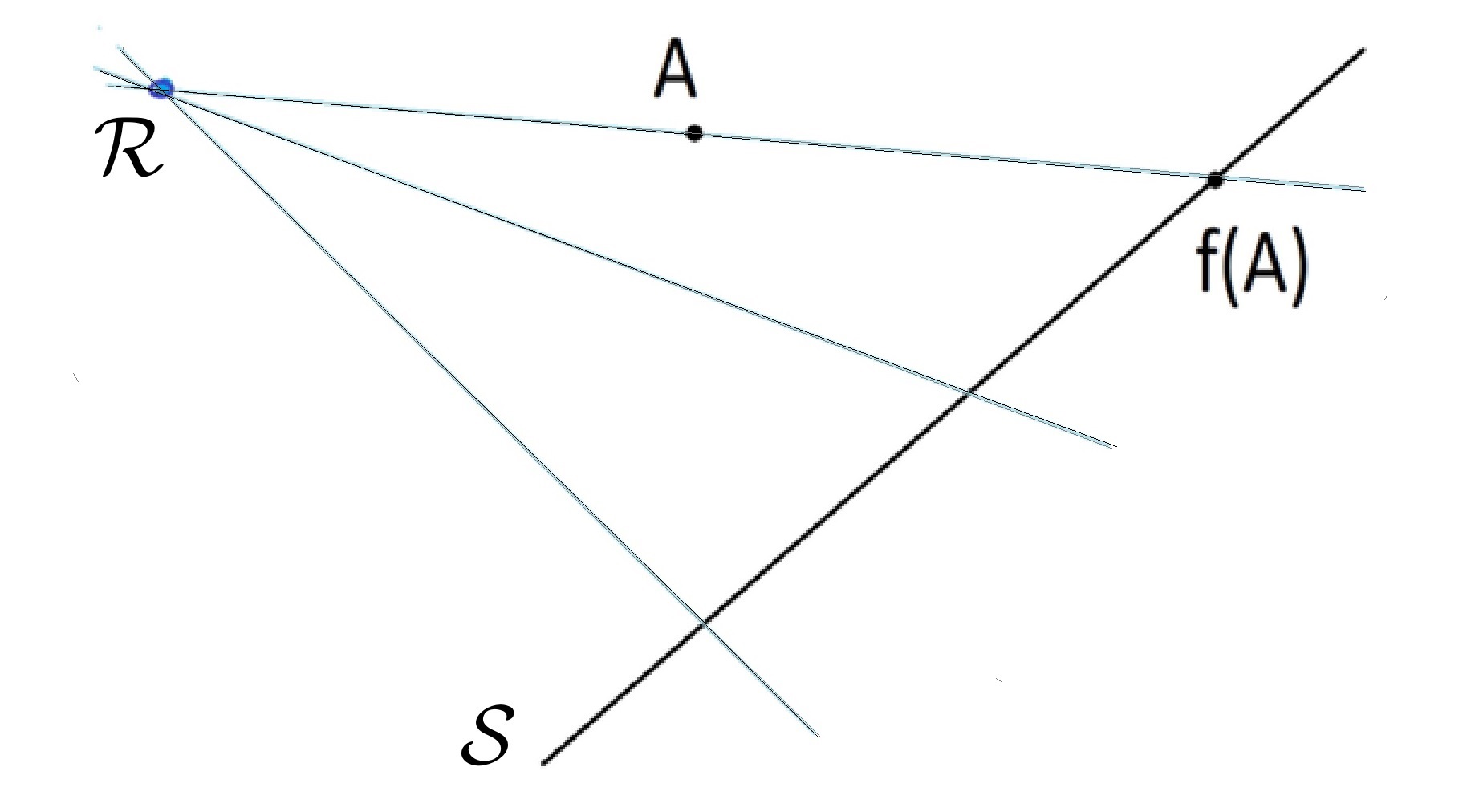}
\caption{}
\label{111}
\end{center}
\end{figure}

Two particular cases of $f_P$ worth mentioning are $rank\hspace{.03cm} P =1$, where the image is the point $P\in\mathbb{C}{\mathrm P}^n$ and 
$rank\hspace{.03cm} P =2$, whose image is the unique projective line $\mathbb{C}{\mathrm P}^1\subset  \mathbb{C}{\mathrm P}^n$ containing $P$ in its convex hull.

\subsection{Geometry of complex curves in $\mathbb{C}{\mathrm P}^n$}  

\hspace{.1cm}

\vspace{.2cm}

Let  $\Sigma$ a compact Riemann surface of genus $g$ and $\varphi:\Sigma\longrightarrow \mathbb{C}{\mathrm P}^n$ be a holomorphic map in the 
complex projective space.
If $\varphi$ is nonconstant then the branch points of the map are isolated and on the unbranched set $\Sigma^\circ$  it defines an immersion.
In this case we say that $\varphi$ is a {\it complex curve}. 
{The curve $\varphi$ is said to be {\it full} } if its image is not contained in any hyperplane $\mathbb{C}{\mathrm P}^{n-1}\subset \mathbb{C}{\mathrm P}^n$.  

\vspace{.1cm}

In this paper  most of the time we think at $\varphi$ as an application into the space of Hermitian matrices and we represent it as
\begin{equation}
A=A_\varphi:\Sigma\longrightarrow \mathbb{C}{\mathrm P}^n \subset {H\hspace{-.03cm}M}_1\hspace{-.03cm}(n+1).
\label{A}
\end{equation}

\vspace{.2cm}

For every  unbranched point $A\in \Sigma^\circ$, the tangent plane could be identified with the corresponding complex 
projective line $\mathbb{C}{\mathrm P}_{\hspace{-.15cm}A}^1$ which is a unit 2-sphere in the Euclidean space. 
We define the {\it Gauss map} $B:{\Sigma^{\circ}} \longrightarrow {H\hspace{-.03cm}M}\hspace{-.03cm}(n+1)$ of the curve 
as the vector joining $A$ with its antipodal point $A^-$ in this $2$-sphere,  $B=A^-\hspace{-.02cm}-\hspace{-.02cm}A$, 
see Figure \ref{tangente}.  
Note that, as a complex line in $\mathbb{C}{\mathrm P}^n$ is determined by two of its points, 
it follows that the vector $B$ determines the tangent $2$-sphere at the point $A$. 
\begin{figure}[h]
\begin{center}
\includegraphics[width=7cm]{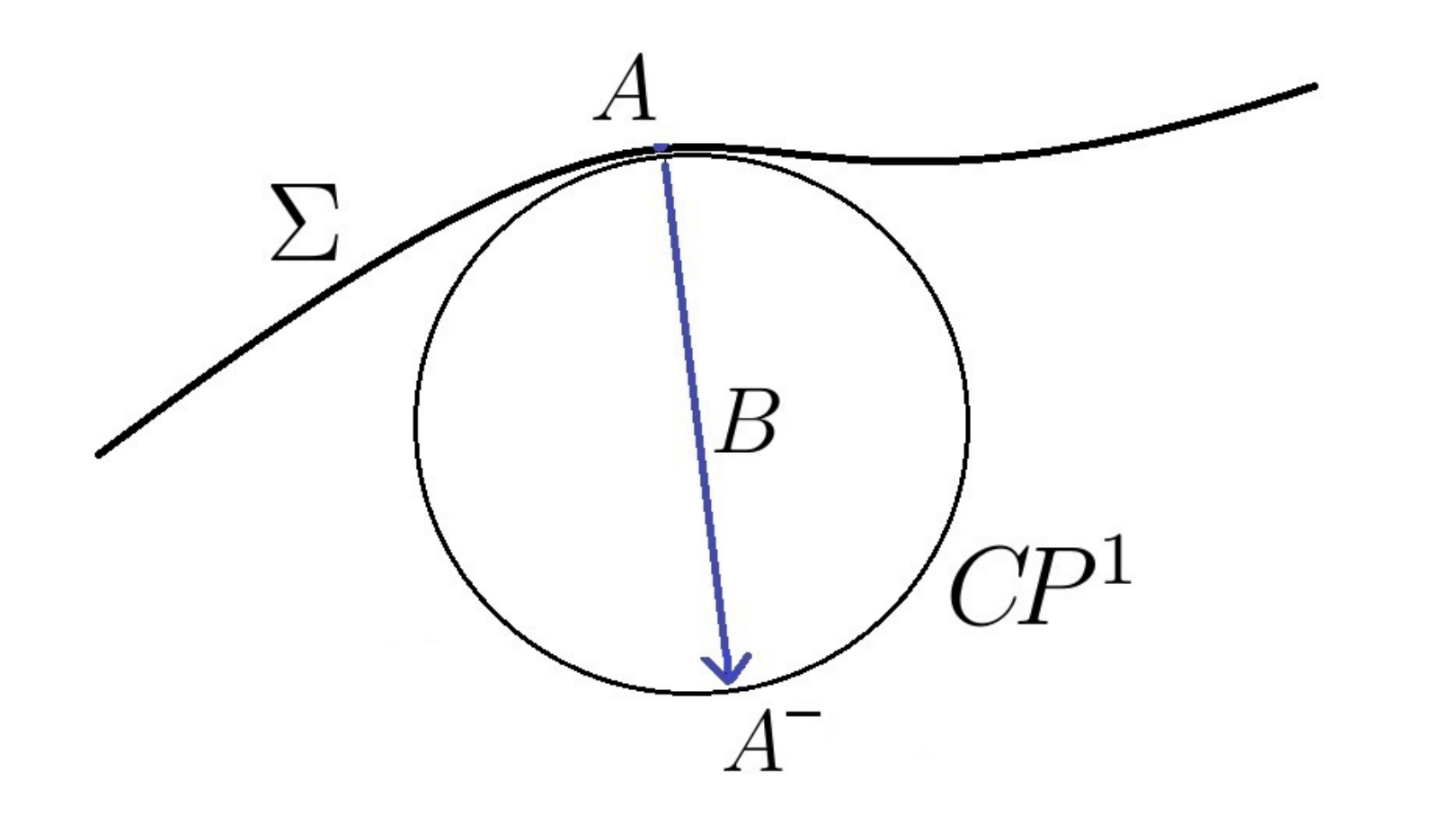}
\caption{}
\label{tangente}
\end{center}
\end{figure}

\vspace{.2cm}

At the unbranched points we will consider, unless otherwise stated, the metric 
$\langle,\rangle$ induced on the surface by the Fubini-Study metric on $\mathbb{C}{\mathrm P}^n$
and we denote by $K$, $d\Sigma$ and 
$\Delta$ the Gauss curvature, the Riemannian measure and Laplacian of the surface.

 \vspace{.2cm}

The complex structure $J$ on $\Sigma$ and the second fundamental form $\sigma$ of  $\Sigma^\circ$ in $\mathbb{C}{\mathrm P}^n$
verifie $\sigma(JX,Y)=\sigma(X,JY)=J\sigma(X,Y)$ and the Gauss equation says that
\begin{equation}
\label{gauss}
K=1-\frac{1}{2}|\sigma|^2.
\end{equation}
The Hessian operator of the vector valued map $A:\Sigma^\circ\longrightarrow  {H\hspace{-.03cm}M}_1\hspace{-.03cm}(n+1)$ will be denoted 
by $\nabla^2\hspace{-.1cm}A$. It coincides with the second fundamental form of $\Sigma^\circ$ in $HM(n+1)$ 
and it can decomposed as
\[
\nabla^2 \hspace{-.1cm}A(X,Y)=\sigma(X,Y)+\tilde{\sigma}(X,Y) \hspace{.8cm} \forall\,  X, Y\in T_A\Sigma.
\]
As complex curves are minimal surfaces in the complex projective space, it follows that the 
Laplacian of the immersion $A:\Sigma^\circ\longrightarrow  {H\hspace{-.03cm}M}\hspace{-.03cm}(n+1)$ is equal to
\begin{equation}
\Delta A =\sum_i \nabla^2\hspace{-.1cm}A(E_i,E_i) 
= \tilde{\sigma}(E_1,E_1)+\tilde{\sigma}(E_2,E_2) = 2\vec{H} = B,
\label{B}
\end{equation}
$E_1,E_2\in T_A\Sigma^\circ$, with $E_2=JE_1$, 
being an orthonormal basis  and $\vec{H}$ the mean curvature vector of the immersion. 

\vspace{.2cm} 

From Lemma 3.2 in  \cite{ros1}, at the points of $\Sigma^\circ$ we have the first rows below
\begin{equation}
\left\{
\begin{array}{ccc}
|I|^2 = 2(n+1)\hspace{1cm} & \langle A,I\hspace{.02cm}\rangle = 2 \hspace{1cm} 
& |A|^2=2 \vspace{.2cm}
\\
\langle B,I \hspace{.02cm}\rangle =0 \hspace{1cm}& \langle B,A\rangle=-2  \hspace{1cm} &|B|^2= 4 
\vspace{.2cm}
\\
\langle \Delta B,I\hspace{.03cm}\rangle=0 \hspace{1cm}& \langle \Delta B,A\rangle= 4  
\hspace{1cm}& \langle \Delta B,B\rangle= - 8 - 2\hspace{.03cm}|\sigma|^2  
\vspace{.2cm}
\\
|\nabla A|^2=2 
\hspace{1cm}
& \langle \nabla A,\nabla B\rangle=-4  \hspace{1cm} & |\nabla B|^2  = 8 +2|\sigma|^2
\end{array}
\right.
\label{calculos}
\end{equation}
and the last one is a consequence of the above. For instance, the Laplacian of the equation $|B|^2=4$ give us 
$|\nabla B|^2 + \langle \Delta B, B\rangle=0$ and then we get
$|\nabla B|^2  = 8 +2\hspace{.03cm}|\sigma|^2.$

\vspace{.3cm}

Now we focus at branch points of the curve $A_\varphi$. 
 Given a nonconstant holomorphic map  \mbox{$A:\Sigma\longrightarrow \mathbb{C}{\mathrm P}^n\subset {H\hspace{-.03cm}M}_1\hspace{-.03cm}(n+1)$} 
and a local complex coordinate  $w, \hspace{.02cm}|w|\hspace{-.03cm}< \hspace{-.02cm}\varepsilon$, around a point $p\in\Sigma$,
using suitable projective coordinates we have 
\[
z= z(w)=\big(1,z_1(w),\ldots,z_n(w)\big) \hspace{-.01cm} \in \hspace{-.01cm} \mathbb{C}^{n+1}\hspace{-.1cm} - \hspace{-.07cm} \{0\} 
\longrightarrow
\mathbb{C}{\mathrm P}^n,     
\]
where $z_i(w)$, $i=1,\ldots,n$, are holomorphic functions. Its relation with the Hermitian matrices is
\begin{equation}
A=A(w) = \frac{1}{|z|^2}\bar{z}^t z . 
\label{A}
\end{equation}
For $w=0$, after a suitable unitary transformation, we have that
\begin{equation}
\left\{
\begin{array}{l}
z(0)=(1,0,\ldots,0) \vspace{.1cm}
\\
A(0)=P_0= 
\left(\hspace{-.12cm}
\begin{array}{cccc}
$\small{1}$\hspace{-.2cm}\vspace{-.2cm}&\hspace{-.3cm}&\hspace{-.3cm}& \hspace{-.1cm}\hspace{-.1cm}
\\
&\hspace{-.06cm}0 && \vspace{-.3cm}\\
&&\hspace{-.3cm}\vspace{-.1cm}\ddots \vspace{-.1cm}& \vspace{-.1cm}\\
&&&\hspace{-.3cm}\vspace{-.2cm}0
\end{array}
\hspace{-.2cm}
\right).
\end{array}
\right.
\label{coordenadas}
\end{equation}
The tangent vector is
\[
z'(w)= \big(0,z'_1(w),\ldots,z'_n(w)\big)\in \mathbb{C}^{n+1} .
\]
If $z'(w)\neq 0$ the point is unbranched and the affine tangent line $\big\{z(w) + \lambda \hspace{.02cm}z'(w) \,/ \lambda \in \mathbb{C}\big\}$
determines the projective tangent one $\mathbb{C}{\mathrm P}_{\hspace{-.15cm}A}^1\subset \mathbb{C}{\mathrm P}^n$.
Moreover, as the Gauss map verifies $B=\Delta A$  it can be expressed at any point of $\Sigma^\circ$ as
\begin{equation}
B= \frac{2}{\big|\partial_{\bar{w}}\partial_{w} A\big|} \partial_{\bar{w}}\partial_{w} A,  \hspace{.3cm}  |w|<\varepsilon,
\label{gauss}
\end{equation}
where we have used that
 the Laplacian $\Delta$ is proportional to the operator $ \partial_{\bar{w}}\partial_{w} $
and the fact that $|B|=2$.

\vspace{.2cm}

If $\varphi$ is nonconstant and $z'(0)=0$, then $p$ is a branch point of the curve. The branch points are isolated and around $w=0$ we have the product
\begin{equation}
z'(w)= w^m h(w),
\label{branching1}
\end{equation}
$h(w)$ being a vectorial holomorphic function with $h(0)\neq 0$, $m\geq 1$ an integer number. 
We say that $m$ is the branching order  at $p$, $b_p(\varphi) =m$. Around this point the induced metric is  $|w|^{2m} \rho^2 |dw|^2$,
where $\rho(0)\neq  0$. The {\it total branching number} of $\varphi$ is defined as the sum of all branching orders
\begin{equation}
b= b(\varphi) = \sum_p b_p(\varphi).
\label{branching}
\end{equation}
The measure $K\hspace{.006cm} d\Sigma$ is smooth, even in the branched case, and the Gauss-Bonnet formula, 
e. g.   Tribuzy \cite{tribuzy}, 
relate the total Gauss curvature with the genus of $\Sigma$ and branching number of the curve
\begin{equation}
\int_\Sigma K \hspace{.02cm}d\hspace{.02cm}\Sigma = 2\pi (2-2g+b).
\label{gauss-bonnet}
\end{equation} 
One other thing worth noting is that the nonzero section $h(w)$, $|w|<\varepsilon$,  in (\ref{branching1})  
allows to extend the tangent projective line of $\varphi$ smoothly through the point $w=0$ and,
therefore,  the Gauss map  is a differentiable on the whole curve, $B:\Sigma\longrightarrow HM(n+1)$.
\begin{lemma}
Let $A:\Sigma\longrightarrow \mathbb{C}{\mathrm P}^n\subset {H\hspace{-.03cm}M}_1\hspace{-.03cm}(n+1)$ be a nonconstant 
holomorphic map and $a\in\mathbb{R}$.  Then
the map $\phi_a:\Sigma\longrightarrow {H\hspace{-.03cm}M}\hspace{-.03cm}(n+1)$, defined as
\begin{equation}
\phi_a = A +a B,
\label{phi}\end{equation}
is smooth everywhere. 
Moreover,  for $0\leq a\leq 1$, the image of $\phi_a$ in contained in the boundary of the convex hull of  $\mathbb{C}{\mathrm P}^n$,  
$\phi_a(\Sigma)\subset  \partial\hspace{.03cm}{\mathcal H}$.
 \end{lemma}
 \begin{proof}
The first part follows from the smoothness of the Gauss map $B$.
The last assertion is a  direct consequence of Figure \ref{tangente}.  
\end{proof}

\vspace{.4cm}

The $degree$ of a nonconstant (possibly branched) complex curve $A:\Sigma\longrightarrow \mathbb{C}{\mathrm P}^n$ is an 
{\it algebro-geometric/topological} invariant given by a positive integer $d$ satisfying the following:

\begin{itemize}
\vspace{.1cm}

\item[$\cdot$] \hspace{.01cm}  The immersed complex curve $A$ is $\mathbb{Z}$-homologous 
to $d$ times the $2$-cycle $\mathbb{C}{\mathrm P}^1\subset \mathbb{C}{\mathrm P}^n$.

\vspace{.1cm}

\item[$\cdot$] \hspace{.01cm}  
With respect to the Fubini-Study, the area of the curve $A:\Sigma\longrightarrow \mathbb{C}{\mathrm P}^n$ is  $Area(\Sigma)=4\pi d$.

\vspace{.1cm}

\item[$\cdot$] \hspace{.01cm}  A general complex hyperplane intersects the immersed curve $\Sigma$ at  $d$ points (counted 
with multiplicity).

\vspace{.1cm}

\item[$\cdot$] \hspace{.01cm} If  $g$ is the genus of $\Sigma$ and $b$ is the total branching number of $A$, the integral of the Gauss curvature  (\ref{gauss}) combined with the Gauss-Bonnet theorem for branched metrics (\ref{gauss-bonnet}) gives
\end{itemize}
\begin{equation}
\int_\Sigma 1 \hspace{.05cm} d\hspace{.03cm}\Sigma= 4\pi d, \hspace{.7cm}
\int_\Sigma |\sigma|^2 d\hspace{.02cm}\Sigma= 8\pi (g+d-1- \frac{1}{2}b).
\label{genus}
\end{equation}
In particular, the total branching number verifies 
\begin{equation}
b /2\leq g+d-1.
\label{beta}
\end{equation}

\vspace{.4cm}

Now we study the behaviour of the vector functions $\phi_a = A + a B$.   

\begin{lemma}   The image of $\phi_a$ lies in a sphere of ${H\hspace{-.03cm}M}_1\hspace{-.03cm}(n+1) $,
\[
\big| \phi_a -\frac{1}{n+1}I\big|^2 = (2a-1)^2 +\frac{n-1}{n+1}
\]
and 
\[
\int_\Sigma |\nabla\phi_a|^2 d\Sigma = 8\pi d\big\{ (2a-1)^2+2a^2\delta\big\},
\]
where $\delta= 1+ \frac{g-1-b/2}{d}\geq 0$.  Furthermore,

\vspace{.1cm}

$(i)$ the energy of $\phi_a$ remains invariant under projective transformations,

\vspace{.1cm}

$(ii)$ the spherical application $\phi_a$ is a conformal map and

\vspace{.1cm}

$(iii)$ the vector-valued function $\phi_a$ belongs to the Sobolev space $W^{1,2}(\Sigma)$.
\label{energy}
\end{lemma}
\begin{proof}
From (\ref{calculos}) we obtain the first equation
\[
\big| \phi_a -\frac{1}{n+1}I\big|^2 = a^2|B|^2+2a\langle B, A-\frac{1}{n+1} \rangle+ |A-\frac{1}{n+1}I|^2 =
\]
\[
4a^2-4a+ \frac{2n}{n+1} = (2a-1)^2+\frac{n-1}{n+1}.
\]

In the same way, at the unbranched points of $A$ we have
\[
\big|\nabla \phi_a\big|^2 = a^2 |\nabla B|^2 +2a\langle \nabla A,\nabla B\rangle + |\nabla A|^2 = (8+2\hspace{.02cm}|\sigma|^2)a^2 -8a +2.
\]
Integrating with respect to $d\hspace{.02cm}\Sigma$ and using (\ref{genus}) 
\[
\int_\Sigma \big|\nabla \phi_a\big|^2 d\hspace{.02cm}\Sigma = 8\pi\big( d+g-1-\frac{1}{2}b\big)a^2 -32\pi d\hspace{.02cm} a + 8\pi d =
\]
\[
8\pi d\big\{ \big( 4 +2\delta \big)a^2 - 4a+1 \big\} = 8\pi d\big\{ (2a-1)^2 +2a^2\delta  \big\}, 
\]
A direct consequence of the above gives $(i)$. 

The assertion in $(ii)$ in proved in \cite{ros2} for $n=2$ and the same computation works for any $n$.

To prove $(iii)$ first observe that the space $W^{1,2}(\Sigma)$ is independent of the smooth metric. Then the assertion follows because
the Dirichlet integral is conformally invariant and that $\phi_a$ is a bounded function.
\end{proof}

\vspace{.2cm}

\subsection{The Brill-Noether Theory.}

\hspace{.1cm}

\vspace{.2cm}

First we recall that the moduli space of closed Riemann surface of genus $g\geq 2$,    ${\mathcal M}_g$,
is a complex algebraic variety of dimension $3g-3$ (with nonempty singular set).
In this context, to say that a property holds for a {\it closed general curve}  means that it happens for all $\Sigma$  in a  certain open dense subset of  
${\mathcal M}_g$.

\vspace{.1cm} 

As a consequence of Brill-Noether theory, see Arbarello, Cornalba, Griffiths and Harris  \cite{ACGH}, Chapter V, we have the following existence result for unbranched full complex curves in term of the genus, degree and dimension of the complex projective space. 

\begin{theorem}[\cite{ACGH}, p.216]  Let $g,n,d\in \mathbb{N}$  with  $n\geq 3$ such that the {\it Brill-Noether constant}  
$\rho =\rho(g,d,n)= g-(n+1)(g-d+n)$ is nonnegative, $\rho\geq 0$. Then a closed general Riemann surface  $\Sigma$ of genus $g$ 
admits a full holomorphic  embedding of degree $d$ in $\mathbb{C}{\mathrm P}^n$,
$\varphi:\Sigma\longrightarrow \mathbb{C}{\mathrm P}^n$. In particular,  $\varphi$ is unbranched.
\label{brill1}
\end{theorem}


The condition that $\rho$ is nonnegative can be rewritten as $d\geq n +g\hspace{.03cm}n/(n+1)$ or, equivalently,
$d\geq d(g,n)$, where
\begin{equation}
d(g,n)= \Big[ \frac{(g+1)n}{n+1}\Big] + n
\label{brill}
\end{equation}
and $[x]$ denotes the integer part of the real number $x$.

\section{The center of mass}

We consider a compact Riemannian surface $(\Sigma,ds^2)$, $d\mu$ its Riemannian measure and 
$A:\Sigma\longrightarrow \mathbb{C}{\mathrm P}^n\subset {H\hspace{-.03cm}M}_1\hspace{-.03cm}(n+1)$ a holomorphic map  of the compact Riemann surface  $\Sigma$ in the complex projective space.

\vspace{.2cm}

Let's recall that, in terms of the coordinate $|w|<\varepsilon$, we have the relations 
\[
A=\frac{1}{|z|^2}\bar{z}^t z  \hspace{.6cm}   {\rm and}  \hspace{.6cm} 
B= \frac{2}{\big|\partial_{\bar{w}}\partial_{w} A\big|} \partial_{\bar{w}}\partial_{w} A,
\]
the first at each point and the second at unbranched points.
Moreover, if $A(w)$ is non constant and the center $w=0$ is a branch point, then the Gauss map $B$ 
extends smoothly through the branching.

\vspace{.4cm}

The points $P\in int \hspace{.03cm}{\mathcal H}$ 
induce another curve
$
A_P=f_P\circ A: \Sigma \longrightarrow \hspace{-.01cm}\mathbb{C}{\mathrm P}^n
$
and, in the coordinates above, 
 \begin{equation}
  A_P = \frac{1}{|zP|^2} \bar{P}^t \bar{z}^t z P ,
  \hspace{1cm}
B_P = \frac{2}{\big|\partial_{\bar{w}}\partial_{w} A_P\big|} 
\partial_{\bar{w}}\partial_{w} A_P , 
\hspace{1cm}
\phi_{P,a}= A_P+a B_P \in \hspace{.02cm}{\mathcal H} .
\label{AB}
\end{equation}

\vspace{.2cm}

 We define {\it the center of mass map} 
$\Phi_a:int\, {\mathcal H}\longrightarrow {H\hspace{-.03cm}M}_1\hspace{-.03cm}(n+1)$,
\begin{equation}
\Phi_a(P)=\frac{1}{Area(ds^2)}\int_\Sigma \big(A_P+ a B_P\big) \, d\mu  \hspace{1.5cm}  P\in int \hspace{.03cm}{\mathcal H}.
\label{definicion}
\end{equation}
The map 
\[
(a,P)\in [0,1)\times  int \hspace{.03cm}{\mathcal H} \longmapsto \Phi_a(P) \in  {H\hspace{-.03cm}M}_1\hspace{-.03cm}(n+1)
\]
is continuous and for any $P\in int\hspace{.03cm}{\mathcal H}$,  it follows from  (\ref{AB}) that 
\begin{equation}
A_P= \bar{P}^t\{\ldots\}P \hspace{1.5cm} B_P= \bar{P}^t\{\ldots\}P,
\label{P}
\end{equation}
where $\{\ldots\}$ indicates a certain, unspecified, square matrix. As the vectors $A_P$ and $B_P$ are bounded, the same holds for the center 
of mass $\Phi_a$ map,
\begin{equation}
\Phi_a(P)= \bar{P}^t\{\ldots\}P \hspace{1.5cm} \forall \, P\in  int \hspace{.03cm}{\mathcal H}.
\label{Phi}
\end{equation}

\vspace{.4cm}

In the rest of the section we will be mostly interested in the case $P \in \partial\hspace{.02cm}{\mathcal H}$, that is, when $P$ is a singular matrix.
First we examine the situation $rank\hspace{0.03cm} P\geq 2$.
      

\begin{lemma} Let $\Sigma$ a closed Riemann surface, $A:\Sigma\longrightarrow \mathbb{C}{\mathrm P}^n$, $n\geq 2$, 
a full branched complex curve and $P\in \partial {\mathcal H}$  with $2\leq rank\hspace{.03cm} P=k +1\leq n$.
Then the map  $A_P$ defines a full branched complex curve in a lineal subspace $ \mathbb{C}{\mathrm P}^k\subset  \mathbb{C}{\mathrm P}^n$, 
\mbox{$A_P: \Sigma\longrightarrow  \mathbb{C}{\mathrm P}^k\subset {H\hspace{-.03cm}M}_1\hspace{-.03cm}(n+1)$.} In particular, $A_P$ is nonconstant, it has finitely many branch points and the associated Gauss map  
$B_P:\Sigma\longrightarrow  {H\hspace{-.03cm}M}\hspace{-.03cm}(n+1)$ is well defined at every point.
\label{rank}
\end{lemma}
\begin{proof} We known from (\ref{proyeccion}) that the projection map $f_P$ is defined outside of a subspace ${\mathcal R}^{n-k}$ and its image is a $k$-dimensional subspace $\mathbb{C}{\mathrm P}^k$. 

As $k\geq 1$, the pullback by $f_P$
of a nontrivial hyperplane in $\mathbb{C}{\mathrm P}^k$ is a non trivial hyperplane on $\mathbb{C}{\mathrm P}^n$. Then, as the initial curve $A$ is full, we get that 
$A_P:\Sigma\rightarrow \mathbb{C}{\mathrm P}^k$ is also full. The rest of the lemma follows directly.   
\end{proof}
Note that the set of branch points of $A_P$ contains the ones of $A$ and some others that appear when we project by the matrix $P$.  
For $k\geq 2$ , $A_P:\Sigma\longrightarrow  \mathbb{C}{\mathrm P}^k$
is a projective algebraic complex curve and when $k=1$ we get a nonconstant meromorphic map  $A_P:\Sigma\longrightarrow  \mathbb{C}{\mathrm P}^1$. 

\vspace{.1cm}

\begin{proposition}Let  $A:\Sigma\longrightarrow \mathbb{C}{\mathrm P}^n$ be a closed full branched complex curve and $0\leq a <1$. Then the behaviour of 
the center of mass map $\Phi_a$
at the boundary of $\mathcal H$ is the following:
 \label{1-6} 
\vspace{.1cm}

In the case $a=0$, 

\vspace{.1cm}

$(i)$ $\Phi_0$ extends to a continuous map $\Phi_0:{\mathcal H}\longrightarrow {H\hspace{-.03cm}M}_1\hspace{-.03cm}(n+1)$,

\vspace{.1cm}

$(ii)$  $\Phi_0(\partial{\mathcal H})\subset \partial{\mathcal H}$ and the restriction $\Phi_0\big|_{\partial{\mathcal H}}$  has non-zero degree, and

\vspace{.1cm}

$(iii)$ $\Phi_0(P)=P$, $\forall\, P\in \mathbb{C}{\mathrm P}^n$.

\vspace{.1cm}

If $0<a<1$ and $n\geq 2$, then 

\vspace{.1cm}

$(iv)$ $\Phi_a$ extends continuously to the complement of 
$\mathbb{C}{\mathrm P}^n$, 
 $\Phi_a:{\mathcal H} \hspace{-.05cm} - \hspace{-.05cm} \mathbb{C}{\mathrm P}^n \longrightarrow {H\hspace{-.03cm}M}_1\hspace{-.03cm}(n+1)$,

\vspace{.1cm}

$(v)$  $\Phi_a\big(\partial{\mathcal H} \hspace{-.05cm} - \hspace{-.05cm} \mathbb{C}{\mathrm P}^n \big)\subset \partial{\mathcal H}$ and 

\vspace{.1cm}

$(vi)$ $\Phi_a$ does not extend, in a continuous way,
to the points of $ \mathbb{C}{\mathrm P}^n$.
\end{proposition}
\begin{proof}
$(i)$ and $(ii)$ are proved in \cite{bourgui}.

\vspace{.1cm}

To prove $(iii)$, using a unitary transformation, it is sufficient to check it for $P=P_0$,
\[
P_0= 
\left(\hspace{-.12cm}
\begin{array}{ccccc}
$\small{1}$\hspace{-.2cm}\vspace{-.2cm}&\hspace{-.3cm}&\hspace{-.3cm}& \hspace{-.1cm}\hspace{-.1cm}
\\
&\hspace{-.3cm}\small{0}\hspace{-.3cm}&\hspace{-.1cm}&& \vspace{-.3cm}
\\
&&\hspace{-.2cm}0 && \vspace{-.3cm}\\
&&&\hspace{-.3cm}\vspace{-.1cm}\ddots \vspace{-.1cm}& \vspace{-.1cm}\\
&&&&\hspace{-.3cm}\vspace{-.2cm}0
\end{array}
\hspace{-.2cm}
\right).
\]

Except for a finite number of points, for any  $[z]=[z_0,\ldots,z_n]\in \Sigma$ we have $z_0\neq 0$. By direct computation
$A_{P_0}:\Sigma\longrightarrow {H\hspace{-.03cm}M}_1\hspace{-.03cm}(n+1)$ is given by
\[
A_{P_0}: [z]\longmapsto \frac{1}{|zP_0|^2}P_0\bar{z}^t z P_0 = \frac{|z_0|^2}{|z_0|^2}P_0=P_0
\]
and therefore
\[
\Phi_0(P_0)=
\frac{1}{Area(ds^2)}\int_\Sigma A_{P_0} d\mu = P_0.
\]  

\vspace{.2cm}

Now we prove  $(iv)$.  
We consider   $P\in \partial\hspace{.02cm}{\mathcal H}$,  $P\notin \mathbb{C}{\mathrm P}^n$,
\mbox{$2\leq rank\hspace{.03cm}P =k+1\leq n+1$.}
 From Lemma \ref{rank} we have that $A_P:\Sigma\longrightarrow  \mathbb{C}{\mathrm P}^k$ is a full complex curve and the 
 map \mbox{$A_P+a B_P: \Sigma \longrightarrow  {H\hspace{-.03cm}M}_1\hspace{-.03cm}(n+1)$}
makes sense at every points of $\Sigma$. Therefore  the integral
\[
\int_\Sigma \big( A_P + a B_P  \big)  \,d \mu  
\]
is well-defined. Since the integrand in the definition of  the center of mass is uniformly bounded, by the dominated
convergence theorem, it follows that the expresion (\ref{definicion}) has a meaning on 
$\partial\hspace{.03cm}{\mathcal H} \hspace{-.05cm} - \hspace{-.05cm} \mathbb{C}{\mathrm P}^n$  and this extension is continuous.  

\vspace{.2cm}

 The assertion $(v)$ follows directly from (\ref{Phi}). 

\vspace{.2cm}

 Finally we show $(vi)$.  We prove it for the matrix $P_0\in\mathbb{C}{\mathrm P}^1$ used in (3). We have seeing that, except in a finite number of points, $A_{P_0}:\Sigma\longrightarrow \mathbb{C}{\mathrm P}^n$ is constantly equal to $P_0$ and $\Phi_0(P_0)=P_0$.

Let's consider  the point $P_\varepsilon \in 
\hspace{.03cm}
\partial \hspace{.02cm} {\mathcal H}$, $\varepsilon >0$ small enough,
given by the matrix
\[
P_\varepsilon= 
\left(\hspace{-.12cm}
\begin{array}{ccccc}
$\small{1}$\hspace{-.2cm}\vspace{-.3cm}&\hspace{-.3cm}&\hspace{-.3cm}& \hspace{-.1cm}\hspace{-.1cm}
\\
&\hspace{-.3cm}\small{\varepsilon}\hspace{-.3cm}\vspace{-.1cm}&\hspace{-.1cm}&& \vspace{-.1cm}
\\
&&\hspace{-.2cm}0 && \vspace{-.3cm}\\
&&&\hspace{-.3cm}\vspace{-.15cm}\ddots \vspace{-.1cm}& \vspace{-.1cm}\\
&&&&\hspace{-.3cm}\vspace{-.1cm}0
\end{array}
\hspace{-.1cm}
\right).
\]
The map $A_{P_\varepsilon}:\Sigma\longrightarrow \mathbb{C}{\mathrm P}^n$ is a nonconstant meromorphic map which consists of projecting
$\Sigma$ onto the complex line \mbox{$\mathbb{C}{\mathrm P}^1=\{ z_2=0,\ldots,z_n=0\}\subset \mathbb{C}{\mathrm P}^n$} and then 
contract it towards $P_0$, 
\[
(z_0,z_1,0,\ldots,0) \mapsto (z_0,\varepsilon z_1,0,\ldots,0).
\]
When $\varepsilon\rightarrow 0$,  outside of finitely many points of $\Sigma$, $A_{P_\varepsilon}$ converges smoothly to the constant  map $A_{P_0}$ and so it follows that the Gauss map vector $B_{P_\varepsilon}$ converges in $L^2(d\mu)$ to the normal vector of  $\mathbb{C}{\mathrm P}^1$ at the point $P_0$.
If we denote this normal vector by $B_{P_0}(\mathbb{C}{\mathrm P}^1)$ we have  that
\[
\lim_{\varepsilon\rightarrow 0} \Phi_a({P_\varepsilon}) =  P_0 + a B_{P_0}(\mathbb{C}{\mathrm P}^1).
\]
Thus the limit depends not only on $P_0$ but also on the limit straight line $\mathbb{C}{\mathrm P}^1$. Finally, if we repeat the argument by using different projective lines passing through $P_0$, we conclude that $\Phi_a$  
does not admit a continuous extension at that point.  
\end{proof}

\begin{proposition}   Let $ds^2$ be a conformal metric on a compact Riemann surface $\Sigma$ and 
\mbox{$A:\Sigma\longrightarrow \mathbb{C}{\mathrm P}^n$}, $n\geq 2$,  be a full complex curve.  
Then, for 
$0\leq a< 1/2$ the point $\frac{1}{n+1}I$ lies in the image of the center of mass map 
$\Phi_a: int\, {\mathcal H}\longrightarrow {H\hspace{-.03cm}M}_1\hspace{-.03cm}(n+1)$.
\label{centro}
\end{proposition}
\begin{figure}[h]
\begin{center}
\includegraphics[width=7.8cm]{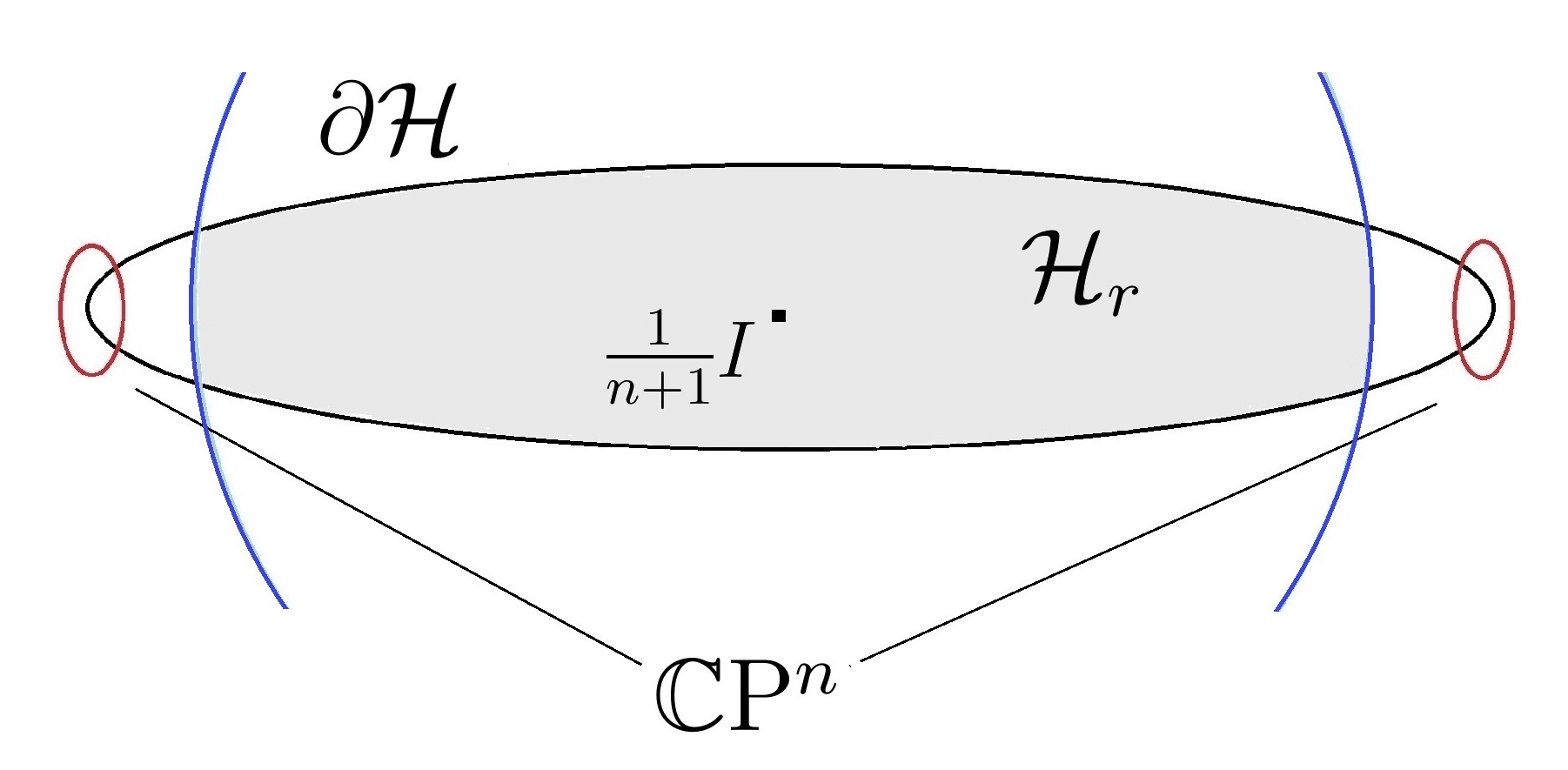}
\caption{${\mathcal H}_r$ is obtained as the intersection of  ${\mathcal H}$ with the ball of radius $r$. The distance between the center
 $\frac{1}{n+1}I$ and $\mathbb{C}{\mathrm P}^n$ is greater than 1 and, when $n$ goes to infinity,
\mbox{$dist( \frac{1}{n+1}I, \partial{\mathcal H}) \rightarrow 0$}.
}
\label{xxxx}
\end{center}
\end{figure}

\vspace{-.3cm}

\begin{proof}  The result is known for $a=0$.
Let's $r$, $1< r < \sqrt{\frac{2n}{n+1}}$ , 
and consider  the convex body
\[
{\mathcal H}_r = 
 {\mathcal H} \cap \bar{B}(\frac{1}{n+1}I,r), 
\]
Figure \ref{xxxx}. Note that the points of $\partial  {\mathcal H}_r$ lie either in  $\partial  {\mathcal H}- \mathbb{C}{\mathrm P}^n$ or in the sphere of center 
$\frac{1}{n+1}I$ and radius $r$.

\vspace{.1cm}

First we prove that
 $\Phi_a(\partial {\mathcal H}_r )$ does not contain the point $\frac{1}{n+1}I$. In fact, if we assume that 
$\Phi_a(P)=\frac{1}{n+1}I$ for some 
$P \in \partial {\mathcal H}_r $, 
from Proposition \ref{1-6}$(v)$
 we have that $P\notin \partial  {\mathcal H}- \mathbb{C}{\mathrm P}^n$. Thus $ P\in int\hspace{.03cm}{\mathcal H}$ and 
\[
|P-\frac{1}{n+1}I|= r.
\]
 
Now, from the identities 
\[
0 =\Phi_a(P)-\frac{1}{n+1}I = \Big( \Phi_0(P)- \frac{1}{n+1}I \Big) + \frac{a}{Area(ds^2)}\int_\Sigma B_P\, d\mu
\]
and (\ref{calculos}) it follows that
\begin{equation}
r= \Big| \Phi_0(P_k)- \frac{1}{n+1}I \Big| \leq  \frac{a}{Area(ds^2)}\int_\Sigma |B_P|\, d\mu =  2a <1,
\label{g}
\end{equation}
which contradicts our choice of $r$. So  $\frac{1}{n+1}I\notin \Phi_a(\partial {\mathcal H}_r )$.

\vspace{.2cm}

We continue with the proof of the result and  divide the remaining argument in two steps.  

\vspace{.2cm}

$\cdot )$ Suppose $a=0$. 
The map $\Phi_0: {\mathcal H}\longrightarrow  {\mathcal H}$ 
is continuous  and \mbox{$\Phi_0:\partial {\mathcal H}\longrightarrow\partial {\mathcal H}$} has non zero topological degree. 
It follows that, for $r$ a bit smaller than $\sqrt{{2n}/({n+1})}$, the image  hypersurface
$\Phi_0(\partial {\mathcal H}_r)$  {\it encloses with nonzero topological degree} the center $\frac{1}{n+1}I$ and this point belongs to the image 
of the domain bounded by the hypersurface:  $\frac{1}{n+1}I\in \Phi_0(int\hspace{.03cm}{\mathcal H}_r)$.

\vspace{.2cm}

$\cdot )$ Now we leave the previous $r$ fixed and we move the parameter $a$ in the interval $0\leq a < 1/2$. The immersed hypersurfaces 
$\Phi_a:\partial{\mathcal H}_r\longrightarrow {H\hspace{-.03cm}M}_1\hspace{-.03cm}(n+1)-\{\frac{1}{n+1}I\}$ 
vary continuously. Then, by using a standard topological argument it follows that, $\forall\hspace{.02cm} a$, the point
$\frac{1}{n+1}I$ belongs to the image by $\Phi_a$ of the domain $int\hspace{.03cm}{\mathcal H}_r$.
This proves the Proposition.
\end{proof}

\section{Upper bounds for the first eigenvalue of the Laplacian}

In this section we estimate the first eigenvalue of the Laplacian of $(\Sigma,ds^2)$, where  $\Sigma$ is a compact Riemann surface of genus $g$ and $ds^2$ a conformal metric on it.  Given a nonconstant holomorphic map $A:\Sigma\longrightarrow \mathbb{C}{\mathrm P}^n$,
at unbranched points, the Riemannian metrics $ds^2$ and $\langle,\rangle$ (the one induced  by the Fubini-Study metric) are conformal
$ds^2=e^{2\theta}\langle,\rangle$, $\theta$ being smooth in $\Sigma^\circ$, and the Riemannian measures are related in the same way 
$d\mu = e^{2\theta}d\hspace{.02cm}\Sigma$. 

\vspace{.2cm}

The first eigenvalue of the Laplacian of the metric $ds^2$ is the larger positive number $\lambda_1(ds^2)$ characterized by the condition
\begin{equation}
\lambda_1(ds^2)\hspace{-.03cm} \int_\Sigma u^2\hspace{.03cm} d\mu \hspace{.03cm} \leq  \int_\Sigma 
|\nabla_{\hspace{-.08cm}ds}\hspace{.02cm}u|^2\hspace{.03cm} d\mu \hspace {.4cm}
\forall \hspace{.03cm} u\in C^1(\Sigma) \hspace{.25 cm} s. \hspace{.04cm}t. \hspace{.2cm} u\neq 0  \hspace{.2cm} {\rm and }\hspace{.05cm}  \int_\Sigma u\hspace{.03cm} d\mu =0,
\label{lambda1}
 \end{equation}
where $|\nabla_{\hspace{-.08cm}ds} \hspace{.02cm}u|$ is the length of the $ds^2$-gradient of $u$. 

\vspace{.2cm}

As test functions we will use the spherical maps 
\mbox{$\phi_a-\frac{1}{n+1}I:\Sigma\longrightarrow \mathbb{C}{\mathrm P}^n \subset\mathbb{R}^{n(n+2)}$}
associated to a complex curve $A:\Sigma\longrightarrow \mathbb{C}{\mathrm P}^n$. 
Therefore the lower part of the inequality in (\ref{lambda1}) transforms to 
$\lambda_1(ds^2)\hspace{-.03cm} Area(ds^2)$.

\vspace{.1cm}

The Dirichlet integral is a conformal invariant and the upper part of  (\ref{lambda1}) could be computed by using the 
induced metric  $\langle,\rangle$. 
In this way we get the energy of $\phi_a$  can be set from the basic invariants of the complex immersion appearing in Lemma \ref{energy}.

\vspace{.1cm}

The tricky point is to be able to get the map $\phi_a-\frac{1}{n+1}I$ to have mean value zero, 
just what we have done in Proposition \ref{centro}.

\vspace{.2cm}

We first prove the following result.
\begin{theorem} Let $\Sigma$ be closed Riemann surface of genus $g$ and 
$A:\Sigma\longrightarrow  \mathbb{C}{\mathrm P}^n$, $n\geq 2$, be a full holomorphic map of degree $d$. Then, for any conformal metric $ds^2$ on $\Sigma$ and $0\leq a\leq 1/2$, 
\[
\lambda_1(ds^2) Area(ds^2)   \leq  8\pi d\left( 1+\frac{2a^2 \delta - \frac{n-1}{n+1}}{(2a-1)^2+ \frac{n-1}{n+1}}   \right),
\]  
where $\delta = 1+\frac{g-1-b/2}{d}$, $b$ being the total branching number of $A$.
\label{teor1}
\end{theorem}
\begin{proof}  Proposition \ref{centro} implies that, by applying a certain projectivity to A,
we can assume that the center of the map of 
$\phi_a:\Sigma\longrightarrow {H\hspace{-.03cm}M}_1\hspace{-.03cm}(n+1)$
is equal to $\frac{1}{n+1}I$. Therefore, from Lemma \ref{energy} and  (\ref{lambda1}) we get
\[
\lambda_1(ds^2) Area(ds^2)   \leq  \frac{\textcolor{white}{...}\int_\Sigma |\nabla \phi_a|^2 d\mu\textcolor{white}{...}}{|\phi_a-\frac{1}{n+1}I|^{{2}}}    =  8\pi d\left( 1+\frac{2a^2 \delta - \frac{n-1}{n+1}}{(2a-1)^2+ \frac{n-1}{n+1}}   \right).
\] 
\end{proof}

\vspace{.2cm}

Note that  (\ref{beta})  implies $0\leq \delta \leq 1+(g-1)/d$. By direct computation we can see that the expresion $F(a)= F(n,d,\delta,a)$, $a\in\mathbb{R}$, where 
\[
F(n,d,\delta, a) = 8\pi d\left( 1+\frac{2a^2 \delta - \frac{n-1}{n+1}}{(2a-1)^2+ \frac{n-1}{n+1}}   \right)
\]
satisfies $F(-\infty)=F(+\infty)$, $F'(0)<0$ and $F'(1/2)>0$. As the rational function $F(a)$ has just two critical points, it follows that it attains its global minimum for a value $a=a_{min}$, with $0<a_{min} <1/2$.

\vspace{.2cm}

\begin{theorem} For any integer numbers $g,n\geq 3$ and $0\leq a \leq 1/2$, one has the eigenvalue inequality
\[
\Lambda_1(g) \leq  8\pi d\left( 1+\frac{2a^2 (1+\frac{g-1}{d})- \frac{n-1}{n+1}}{(2a-1)^2+ \frac{n-1}{n+1}}   \right),
\]
where  $d$ is a positive integer with
\[
d \geq d(g,n) = \left[\frac{(g+1)n}{n+1}\right] + n.
\]
\label{teor2}
\end{theorem}
\begin{proof}
Brill-Noether's technique in Theorem \ref{brill1} says that for a {\it general Riemann surface} $\Sigma$ with $g\geq 3$ there is a  full
holomorphic map $A:\Sigma\longrightarrow  \mathbb{C}{\mathrm P}^n$ of degree $d \geq d(g,n)$ and branching number $b=0$. So, 
for any conformal metric $ds^2$ on $\Sigma$,  from Theorem \ref{teor1} we obtain
 \begin{equation}
\lambda_1(ds^2) Area(ds^2)   \leq  8\pi d\left( 1+\frac{2a^2 (1+\frac{g-1}{d})- \frac{n-1}{n+1}}{(2a-1)^2+ \frac{n-1}{n+1}}   \right).   
\label{brill2}
\end{equation}
To get the result for any metric $ds^2$ on a surface of genus $g$, we observe that, in the space of smooth metrics on a compact surface, 
a neighborhood of $ds^2$ provides an open family of conformal structures. Therefore, the continuity of the first eigenvalue functional give the 
inequality (\ref{brill2}) not only for metrics on a {\it general Riemann surface}, but for arbitrary metrics on the surface of this topology.
\end{proof}

\begin{theorem} The asymptotic growth of the sequence $\{\Lambda_1(g)\}$ verifies the inequality
\[
\limsup_{g\rightarrow \infty} \frac{1}{g}\Lambda_1(g) \leq 4(3-\sqrt{5})\pi \approx 3.056\pi.
\]
\label{teor3}
\end{theorem}
\begin{proof}
For any positive integer $g$ we consider the values 
\[
d=g+1, \hspace{.6cm}  n=\left[\sqrt{g+1}\right].
\]
The  Brill-Noether constant is  $\rho = g - (n+1)(g-d+n) = g - (n + 1)(n - 1) = g+1-n^2 \geq 0$ and, so, $d\geq d(g,n)$.
From Theorem  \ref{teor2}, we have
\[
\frac{1}{g}\Lambda_1(g) \leq  8\pi \frac{d}{g}\left( 1+\frac{2a^2 (1+\frac{g-1}{d})- \frac{n-1}{n+1}}{(2a-1)^2+ \frac{n-1}{n+1}}   \right).  
\]
Now we take limits
\[
\lim_{g\rightarrow\infty}\,  \frac{ d}{g} =1, \hspace{1cm}\lim_{g\rightarrow\infty} \,  \frac{n-1}{n+1} =1
\]
and we get that,  for any $a$ between $0$ and $1/2$,
\[
\limsup_{g\rightarrow \infty}\,  \frac{1}{g}\Lambda_1(g) \leq 8\pi \left( 1 + \frac{4a^2-1}{(2a-1)^2+1}  \right). 
\]
If we call $G(a)$ the expression at the right,  by direct calculation
we see that its minimum value is attains for 
\[
a_{min} = \frac{3-\sqrt{5}}{4} \approx 0.191  \hspace{.5cm} {\rm and}  \hspace{.5cm} G(a_{min}) = 4(3-\sqrt{5})\pi \approx 3.056\pi.
\]
This proves the theorem.
\end{proof}


\vspace{.1cm}

A natural variant is to look at the asymptotic behaviour of the first eigenvalue for some families of Riemannian surfaces. 
The one that has attracted the most interest is that of hyperbolic metrics, Buser, Burger and Dodziuk \cite{bbd}.
Let $\Lambda_1(g,-1) $ be the supremum of the first eigenvalue of the Laplacian on compact surfaces of genus $g$ 
and curvature $K\equiv -1$. 
The limit upper bound of the sequence was known to be $\leq 0.25$ $\big($Cheng \cite{cheng}, Huber \cite{huber}$\big)$ and
 Hide and Magee \cite{hide} prove the equality
\[
\limsup_{g\rightarrow \infty} \,  \Lambda_1(g,-1)  = \frac{1}{4}.
\]
 For others asymptotic properties 
 see e. g.   Lipnowski and Wright \cite{lip} and Wu and Xue \cite{wx}.

{\footnotesize
\noindent
{Antonio Ros,} { aros@ugr.es}  \vspace{-.10cm} \\
Department of Geometry and Topology and \vspace{-.10cm}\\
Institute of Mathematics (IMAG), \vspace{-.10cm}\\
University of Granada\vspace{-.10cm}\\
 18071 Granada, Spain.
 }

\end{document}